\documentclass[11pt,leqno]{amsart}
\usepackage[a4paper,top=2.5 cm,bottom=2 cm,left=2.5 cm,right=2 cm]{geometry}
\usepackage{graphicx}
\usepackage{upref}
\usepackage{textcomp} 
\usepackage{tikz} 
\usepackage{amsmath ,amssymb}
\usepackage{color,bbm,epsfig,float,fancyhdr}
\usepackage{caption}
\DeclareMathOperator{\sgn}{sgn}

\newtheorem{definition}{Definition}[section]
\newtheorem{theorem}{Theorem}[section]
\newtheorem{lemma}{Lemma}[section]

\newtheorem{remark}{Remark}[section]
\newtheorem*{maintheorem*}{Main Theorem}
\allowdisplaybreaks
\numberwithin{equation}{section}

\renewcommand{\i}{\ifmmode\mathit{\mathchar"7010 }\else\char"10 \fi}
\renewcommand{\j}{\ifmmode\mathit{\mathchar"7011 }\else\char"11 \fi}
\newcommand{\R}{\mathbb{R}}
\newcommand{\N}{\mathbb{N}}

\newcommand{\abs}[1]{\left|#1\right|}
\newcommand{\modulo}[1]{\left|#1\right|}
\newcommand{\norm}[1]{\left\|#1\right\|}

\newcommand{\norma}[1]{{\left\|#1\right\|}}

\newcommand{\pt}{\partial_t}
\newcommand{\px}{\partial_x}

\newcommand{\pxx}{\partial_{xx}^2}

\newcommand{\pxxx}{\partial_{xxx}^3}

\newcommand{\ptx}{\partial_{tx}^2}

\newcommand{\Cc}[1]{\mathbf{C_c^{#1}}}

\newcommand{\vfi}{\varphi}

\newcommand{\eps}{\varepsilon}
\newcommand{\epk}{{\varepsilon_k}}

\newcommand{\re}{\rho_\varepsilon}
\newcommand{\ve}{\mathfrak{v_\varepsilon}}

{%

\begin{enumerate}}%
{\end{enumerate}}

{%

\begin{enumerate}}%
{\end{enumerate}}

\begin{document}\large

\title[Non-local discontinuous flux]{Non-local scalar conservation laws\\ with discontinuous flux}
\date{\today}

\author{F. A. Chiarello} 
\address[Felisia Angela Chiarello]{\newline Department of Mathematical Sciences “G. L. Lagrange”, Politecnico di Torino, Corso Duca degli Abruzzi 24, 10129 Torino, Italy.}
\email[]{felisia.chiarello@polito.it}

\author{G. M. Coclite}
\address[Giuseppe Maria Coclite]{\newline
Department of Mechanics, Mathematics and Management, Polytechnic University of Bari, Via E.~Orabona 4, 70125 Bari, Italy}
\email[]{giuseppemaria.coclite@poliba.it}



\thanks{The authors wish to thank Paola Goatin for her suggestions and valuable discussions.\\GMC is member of the Gruppo Nazionale per l'Analisi Matematica, 
la Probabilit\`a e le loro Applicazioni (GNAMPA) of the Istituto Nazionale di Alta Matematica (INdAM). \\	
FAC acknowledges support from “Compagnia di San Paolo” (Torino, Italy)}

\begin{abstract}
We prove the well-posedness of entropy weak solutions for a class of space-discontinuous scalar conservation laws with non-local flux arising in traffic modeling. We approximate the problem adding a viscosity term and we provide $L^\infty$ and BV estimates for the approximate solutions. We use the doubling of variable technique to prove the stability with respect to the initial data from the entropy condition. 
\end{abstract}

\maketitle
\section{Introduction}
\label{sec:introduction}
The first macroscopic traffic flow model, based on fluid-dynamics equations, is the Lighthill, Whitham and Richards (LWR) model \cite{LighthillWhitham, Richards}. 
It consists in one scalar equation that expresses the conservation of the number of cars. One shortcoming of the LWR model is that does not match the experimental data because  its main assumption is that  the mean traffic velocity is a function of the traffic density, but this is not  true in real congested regimes.  Another limitation of this model is that allows for infinite acceleration of cars. For these reasons, \textit{non-local} versions of the LWR model have been proposed in \cite{BlandinGoatin2016, ChiarelloGoatin, SopasakisKatsoulakis2006}. In these models, the speed depends on a weighted mean of the downstream traffic density and, as a consequence, it becomes a Lipschitz function with respect to space and time variables, overcoming the limitation of classical macroscopic models that allows for speed discontinuities. 
Non-local traffic models describes the behaviour of drivers that adapt their velocity with respect to what happens to the cars in front of them. In particular, the flux function depends on a \textit{downstream} convolution term between the density of vehicles and a kernel function with support on the negative axis. See \cite{Chiarello} for an overview about non-local traffic models.
Recently, in \cite{chiarelloFriedrichGoatinGK}, the authors introduced a model focusing on a non-local mean downstream velocity, describing the behavior of drivers on two stretches of a road with different velocities and capacities, without violating the maximal density constraint on each road segment. It is worth underlying that in \cite{chiarelloFriedrichGoatinGK}, the authors consider a continuous flux function to model a 1-to-1 junction. On the contrary, in this manuscript we propose a non-local scalar space discontinuous model to describe the beaviour of drivers on two consecutive roads with different  speed limits. Traveling waves for this kind of model are studied in \cite{Shen2018}.
Here, we prove the well-posedness of a non-local space discontinuous  traffic model and our approach is based on a viscous regularizing approximation of the problem and standard compactness estimates.  The paper is organized as follows. 
In Section \ref{sec:MainResults}, we present a non-local discontinuous problem and describe the main results in this paper. In Section \ref{sec:exist}, we prove the existence of weak solutions of our problem, approximating it through a viscous problem and giving $L^\infty$ and $BV$ bounds. Finally, in Section \ref{sec:uniq}, we show the uniqueness of entropy solutions, deriving an $L^1$ contraction property using  a doubling of variables argument. 

\section{Main results} \label{sec:MainResults}
We consider the following scalar conservation equation with discontinuous non-local flux coupled with an initial datum
\begin{equation}
\label{eq:CL}
\begin{cases}
\pt \rho + \px f(t,x,\rho)=0, &\quad (t,x)\in (0,\infty)\times \R,\\ 
\rho(0,x)=\rho_0(x), &\quad x\in \R,
\end{cases}
\end{equation}
where
\begin{align*}
	&f(t,x,\rho)=\rho (1-w_\eta*\rho)\mathfrak{v}(x),\\
	&(w_\eta*\rho)(t,x)= \int_{x}^{x+\eta} \rho(t,y) w_\eta(y-x) dy,\qquad \eta>0,
\end{align*}
and  the velocity function $\mathfrak{v}=\mathfrak{v}(x)$ is defined as follows
\begin{equation*}
	\mathfrak{v}(x)=\begin{cases}
		v_l,& \quad \text{if $x<0$},\\
		v_r,& \quad \text{if $x>0$}.
	\end{cases}
\end{equation*}
In this context $\rho$ represents the density of vehicles on the roads, $\omega_\eta$ is a non-increasing kernel function whose support $\eta$ is proportional to the look-ahead distance of drivers, that are supposed to adapt their velocity with respect to the mean downstream traffic density.

The equation in \eqref{eq:CL} is a non-local version of the Lightill-Whitham-Richards traffic model \cite{GP,LighthillWhitham, Richards} with a discontinuous velocity field \cite{CR,KR}.

On $w_\eta,\, \mathfrak{v},\,\rho_0$ we shall assume that
\begin{align}
\label{ass:v} &0<v_l< v_r;\\
\label{ass:w} w_\eta \in C^{2}([0,\eta]),\qquad w_\eta(\eta)=w'_\eta(\eta)=0, \qquad 
&\omega'_\eta\le0\le w_\eta,\qquad\norm{w_\eta}_{L^1(0,\eta)}=1;\\
\label{ass:init} 0\le \rho_0\le 1,\qquad  &\rho_0\in L^1(\R)\cap BV(\R).
\end{align}

Assumption \eqref{ass:w} implies that, if $\rho$ is continuos,
\begin{equation}
\label{eq:w'}
\begin{split}
\px(\omega_\eta*\rho)(t,x)&= -(\omega_\eta' *\rho)(t,x)-\omega_\eta(0)\rho(t,x),\\
\px(\omega_\eta'*\rho)(t,x)&= -(\omega_\eta'' *\rho)(t,x)-\omega_\eta'(0)\rho(t,x).
\end{split}
\end{equation}

\begin{remark}
	Assumption \eqref{ass:v} is motivated by the Maximum Principle. Indeed, if \eqref{ass:v} does not hold, namely $v_l>v_r,$ 
we cannot say that 
$$\norm{\rho}_{L^\infty((0,\infty)\times\R)}\le\norm{\rho_0}_{L^\infty(\R)},$$ 
see \cite{ChiarelloVillada} for numerical evidences in the non-local case.
 Let us consider this very easy example in the classical case
\begin{align}\label{eq:counterex}
\begin{cases}
\partial_t \rho + \partial_x f(\rho)=0, &\quad (t,x)\in (0,\infty)\times (-\infty,0),\\ 
\partial_t \rho +  \partial_x g(\rho)=0, &\quad (t,x)\in (0,\infty)\times (0,\infty),\\
\rho(0,x)=\rho_0(x), &\quad x\in \R,
\end{cases}
\end{align}
where 
\begin{equation*} 
f(\rho)=2 (\rho (1-\rho)),\quad g(\rho)=\rho (1-\rho),\quad \rho_0(x)=
\begin{cases}
0.25,& \quad \text{if $x<0$},\\
0.77,& \quad \text{if $x>0.$}
\end{cases}
\end{equation*}
The entropy weak solution to the above Cauchy problem is
\begin{equation}
\rho=
\begin{cases}
\rho_l=0.25,& \quad \text{if $x<\frac{f(\rho_-)-f(\rho_l)}{\rho_--\rho_l}t,$}\\
\rho_-=0.9,& \quad \text{if $\frac{f(\rho_-)-f(\rho_l)}{\rho_--\rho_l}t<x<0,$}\\
\rho_+=\rho_r=0.77,& \quad \text{if $x>0.$}
\end{cases}
\end{equation}
\begin{center}
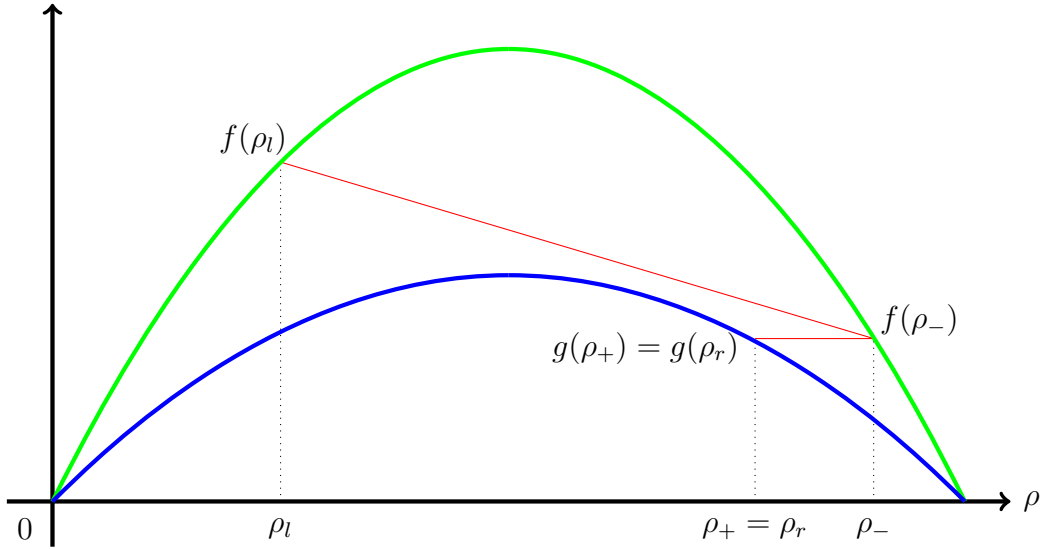

\begin{tikzpicture}[scale=12]
\draw[->, ultra thick] (0,-0.05)--(0,0.55); 
\draw[->,ultra thick] (-0.05,0)--(1.05,0) node[right] {$\rho$};
 \draw[dotted] ( 0.25 ,0) - -(0.25 ,0.375);
  \draw[dotted] ( 0.77 ,0) - -(0.77 ,0.18);
   \draw[dotted] ( 0.9 ,0) - -(0.9 ,0.18);
    \draw[red] ( 0.77 ,0.18) - -(0.9 ,0.18);
     \draw[red] ( 0.25 ,0.375) - -(0.9 ,0.18);
      \node at (-0.03,-0.03) {$0$};
   \node at (0.22,0.4) {$f(\rho_l)$};
    \node at (0.95,0.2) {$f(\rho_-)$};
     \node at (0.65,0.17) {$g(\rho_+)=g(\rho_r)$};
      \node at (0.25,-0.03) {$\rho_l$};
       \node at (0.77,-0.03) {$\rho_+=\rho_r$};
         \node at (0.9,-0.03) {$\rho_-$};
       \draw[ultra thick,green](0.5,.5) parabola (0,0);
\draw[ultra thick,green](0.5,.5) parabola (1,0);
 \draw[ultra thick,blue](.5,0.25) parabola (0,0);
 \draw[ultra thick,blue] (.5,0.25)   parabola (1,0) ;
\end{tikzpicture}
\captionof{figure}{Fundamental diagrams relative to \eqref{eq:counterex}.}\label{fi:diagrams}
\end{center}

A complete description of conservation laws with discontinuous flux can be found in \cite{GNPT,KR}. 

\end{remark}
We use the following definitions of solution.

\begin{definition}
\label{def:sol}
We say that a function $\rho:[0,\infty)\times\R\to\R$ is a weak solution of \eqref{eq:CL} if
\begin{equation}
0\le \rho\le 1,\qquad \norm{\rho(t,\cdot)}_{L^1(\R)}\le \norm{\rho_0}_{L^1(\R)},
\end{equation}
for almost every $t>0$ and for every test function $\vfi\in \Cc1(\R^2)$ 
\begin{equation*}
\int_0^\infty\int_\R\left( \rho \pt\vfi+ f(t,x,\rho)\px\vfi\right)dtdx +\int_\R \rho_0(x)\vfi(0,x)dx=0.
\end{equation*}
\end{definition}

\begin{definition}
	A function $\rho\in(L^1\cap L^\infty)(\R^+\times\R; [0,\rho_{\max}])$ is an entropy weak solution of \eqref{eq:CL}, if 
\begin{itemize}
\item[1)] for all $\kappa\in\R,$ and any test function $\varphi\in\Cc1(\R^2;\R^+)$ which vanishes for $x\leq0,$ 
\begin{align*}
	&	\int_0^{+\infty}\int_{\R^+} \abs{\rho-\kappa}\varphi_t+\abs{\rho-\kappa}(1-\omega_\eta\ast \rho)v_r \varphi_x\\
	&\hspace{5em}	-\sgn{(\rho-\kappa)}\,\kappa\, \px(\omega_\eta\ast\rho)v_r\varphi\, dx\,dt+ \int_{\R^+} \abs{\rho_0(x)-\kappa}\varphi(0,x) dx \geq0;
\end{align*}\\
\item[2)] for all $\kappa\in\R,$ and  any test function $\varphi\in\Cc1(\R^2;\R^+)$ which vanishes for $x\geq0,$ 
\begin{align*}
&	\int_0^{+\infty}\int_{\R^-} \abs{\rho-\kappa}\varphi_t+\abs{\rho-\kappa}(1-\omega_\eta\ast \rho)v_l\, \varphi_x\\
&\hspace{5em}	-\sgn{(\rho-\kappa)}\,\kappa\, \px(\omega_\eta\ast\rho)\,v_l\,\varphi\, dx\,dt+ \int_{\R^-} \abs{\rho_0(x)-\kappa}\varphi(0,x) dx \geq0;
\end{align*}
\item[3)]
for all $\kappa\in\R,$ and  any test function $\varphi\in\Cc1(\R^2;\R^+)$ 
\begin{align*}
	&	\int_0^{+\infty}\int_{\R} \abs{\rho-\kappa}\varphi_t+\abs{\rho-\kappa}(1-\omega_\eta\ast \rho)v(x)\, \varphi_x\\
	&\hspace{5em}	-\int_0^{+\infty}\int_{\R^* }\sgn{(\rho-\kappa)}\,\kappa\, \px\left((\omega_\eta\ast\rho)\,\mathfrak{v}(x)\right)\,\varphi\, dx\,dt+ \int_{\R} \abs{\rho_0(x)-\kappa}\varphi(0,x) dx \\
	&\hspace{5em}+ \int_{0}^{+\infty}  \abs{(v_r-v_l) \kappa (1-w*\rho)} \varphi(t,0) dt\geq0;
	\end{align*}
\item[4)]   the traces are such that the jump
\begin{equation}
\abs{\rho_l-\rho_r}
\end{equation}
is the smallest possible that  
satisfies the Rankine-Hugoniot condition\\
\begin{equation*}
f(t,0^+,\rho_r)=f(t,0^-,\rho_l) \text{ i.e. }
v_l \rho_l=v_r \rho_r,
\end{equation*}
 where we denoted with $f(t,0^\pm,\rho_{r,l})=\lim_{x\to0^\pm} \mathfrak{v}(x) \rho(t,x) \int_{x}^{x+\eta}\omega_\eta(y-x)\rho(y,t) dy$.
 \end{itemize}
\end{definition}

\begin{remark}
	We would like to underline that the existence of strong right and left traces, respectively $\rho_r$ and $\rho_l$, is ensured by the genuine non-linearity of our flux function, as it is proved in \cite{AleMitrovic2013,AndreianovMitrovic}. 
\end{remark}

The main result of this paper is the following.

\begin{theorem}
\label{th:main}
Assume  \eqref{ass:v}, \eqref{ass:w}, and \eqref{ass:init}. Then, the initial value problem \eqref{eq:CL} possesses an unique entropy solution $u$ in the sense of Definition 2.2. 
Moreover, if $u$ and $v$ are two entropy solutions of  \eqref{eq:CL} in the sense of Definition 2.2 
, the following inequality holds
 \begin{equation}
 \label{eq:stability}
\norm{u(t,\cdot)-v(t,\cdot)}_{L^1(\R)}\le  e^{K(T) t}\norm{u(0,\cdot)-v(0,\cdot)}_{L^1(\R)},
\end{equation}
for almost every $0<t<T$, $R>0$, and some suitable constant $K(T)>0$.
\end{theorem}

\section{Existence}
\label{sec:exist}

Our existence argument is based on passing to the limit in a vanishing viscosity approximation of \eqref{eq:CL}.

Fix a small number $\eps>0$ and let $\re=\re(t,x)$ be the unique classical solution of the following problem
\begin{equation}
\label{eq:CLeps.2}
\begin{cases}
\pt \re + (1-w_\eta*\re)\ve(x) \px \re+\re |1-w_\eta*\re| \ve'(x)&{}\\
\qquad \qquad+\re (w_\eta'*\re) \ve(x)+\re^2 w_\eta(0)\ve(x)=\varepsilon \pxx \re, &\quad (t,x)\in (0,\infty)\times \R,\\ 
\re(0,x)=\rho_{0,\eps}(x), &\quad x\in \R,
\end{cases}
\end{equation}
where $\rho_{0,\eps}$ and $\ve$ are $C^\infty(\R)$ approximations of $\rho_0$ and $v$ such that
\begin{equation}
\label{eq:init-v-eps}
\begin{split}
&\rho_{0,\eps}\to \rho_{0},\quad \text{a.e. and in $L^p(\R),\, 1\le p< \infty$},\\
&0\le \rho_{0,\eps}\le 1,\qquad \quad \norm{\rho_{0,\eps}}_{L^1(\R)}\le \norm{\rho_{0}}_{L^1(\R)},\quad \norm{\px \rho_{0,\eps}}_{L^1(\R)}\le C_0\\
&v_l\le\ve\le v_r,\quad \ve'\ge0,\quad \ve(x)=\begin{cases}
v_l&\text{if $x<-\eps$},\\
v_r&\text{if $x>\eps$},
\end{cases}
\end{split}
\end{equation}
for every $\eps>0$ and some positive constant $C_0$ independent on $\eps$.
The well-posedness of \eqref{eq:CLeps.2} can be obtained following the same arguments of \cite{CC,CdR,CHK}.

Let us prove some a priori estimates on $\re$ denoting with $C_0$ the constants which depend only on the initial data, and $C(T)$ the constants which depend also on $T$.

\begin{lemma}[{\bf $L^\infty$ estimate}]
\label{lm:linfty}
We have that
\begin{equation*}
0\le \re\le 1,
\end{equation*}
for every $\eps>0$.
\end{lemma}

\begin{proof}
Thanks to \eqref{eq:init-v-eps}, $0$ is a subsolution of \eqref{eq:CLeps.2}, due to the Maximum Principle for parabolic equations we have that
\begin{equation}
\label{eq:pos}
 \re\ge 0.
\end{equation}

We have to prove
\begin{equation}
\label{eq:le1}
 \re\le 1.
\end{equation}

Assume by contradiction that \eqref{eq:le1} does not hold.

Let us define the function $r(t,x)=e^{-\lambda t}\re(t,x).$ 
We can choose $\lambda$ so small that
\begin{equation}
\label{eq:CLeps.4.0.1}
\norm{r}_{L^\infty((0,\infty)\times\R)}>1.
\end{equation}

Thanks to \eqref{eq:CLeps.2}, $r$ solves the equation
\begin{equation}
\label{eq:CLeps.4.0}
\begin{split}
\pt r +\lambda r &+ (1-w_\eta*\re)\ve(x) \px r-\varepsilon \pxx r\\
=&-r(w_\eta'*\re)\ve(x)-r|1-w_\eta*\re| \ve'(x)- e^{\lambda t}r^2  w_\eta(0)\ve(x).
\end{split}
\end{equation}
Since 
\begin{equation*}
(\omega_\eta'\ast\re(t,\cdot))(x)=\int_{x}^{x+\eta}\omega_\eta'(y-x)\left(\re(t,y)-\norm{\re}_{L^\infty((0,\infty)\times\R)}\right)dy-
\norm{\re}_{L^\infty((0,\infty)\times\R)} \omega_\eta(0), 
\end{equation*}
we can write 
\begin{equation}
\begin{split}
\pt r +\lambda r &+ (1-w_\eta*\re)\ve(x) \px r-\varepsilon \pxx r\\
=&-r\left(w_\eta'*\left(\re-\norm{\re}_{L^\infty((0,\infty)\times\R)}\right)\right)\ve(x)+r \omega_\eta(0)\norm{\re}_{L^\infty((0,\infty)\times\R)}\ve(x)\\
&-r|1-w_\eta*\re| \ve'(x)- e^{\lambda t}r^2  w_\eta(0)\ve(x)\\
=&-r\left(w_\eta'*\left(\re-\norm{\re}_{L^\infty((0,\infty)\times\R)}\right)\right)\ve(x)-r|1-w_\eta*\re| \ve'(x)\\
&+r(\norm{\re}_{L^\infty((0,\infty)\times\R)}-\re) w_\eta(0)\ve(x)\leq 0.
\end{split}
\end{equation}
Let $(\bar t, \bar x)$ be such that 
\begin{equation*}
\norm{r}_{L^\infty((0,\infty)\times\R)}=r(\bar t,\bar x).
\end{equation*} 
Since, thanks to \eqref{eq:CLeps.4.0.1},
\begin{equation*}
\norm{r(0,\cdot)}_{L^\infty(\R)}\le 1<r(\bar t,\bar x),
\end{equation*} 
we must have
\begin{equation*}
\bar t>0.
\end{equation*} 
Therefore we can evaluate  \eqref{eq:CLeps.4.0} in $(\bar t, \bar x)$ and gain 
\begin{equation*}
0<\lambda \norm{r}_{L^\infty((0,\infty)\times\R)}\leq 0.
\end{equation*} 
Since, this cannot be, \eqref{eq:le1} is proved.
 \end{proof}

Using \eqref{ass:w} and Lemma \ref{lm:linfty}, we know that
\begin{equation}
\label{eq:linftyconv}
0\le w_\eta*\re\le1
\end{equation} 
and then we can rewrite \eqref{eq:CLeps.2} as follows
\begin{equation}
\label{eq:CLeps}
\begin{cases}
\pt \re + \px (\re (1-w_\eta*\re)\ve(x))=\eps\pxx\re, &\quad (t,x)\in (0,\infty)\times \R,\\ 
\re(0,x)=\rho_{0,\eps}(x), &\quad x\in \R.
\end{cases}
\end{equation}

\begin{lemma}[{\bf $L^1$ estimate}]
\label{lm:l1}
We have that
\begin{align}
\label{eq:l1}
\norm{\re(t,\cdot)}_{L^1(\R)}\le& \norm{\rho_0}_{L^1(\R)},\\
\label{eq:l1conv}
\norm{(w_\eta*\re)(t,\cdot)}_{L^1(\R)}\le& \norm{\rho_0}_{L^1(\R)},\\
\label{eq:BVconv}
\norm{\px(w_\eta*\re)(t,\cdot)}_{L^1(\R)}\le& 2w_\eta(0)\norm{\rho_0}_{L^1(\R)},
\end{align}
for every $t\ge0$ and $\eps>0$.
\end{lemma}

\begin{proof}
We have 
\begin{align*}
\frac{d}{dt}\int_\R \re dx=&\int_\R\pt\re dx
=\eps\int_\R \pxx\re dx-\int_\R \px (\re (1-w_\eta*\re)\ve(x))dx=0.
\end{align*}
Therefore,
\begin{equation*}
\norm{\re(t,\cdot)}_{L^1(\R)}= \norm{\rho_{0,\eps}}_{L^1(\R)},
\end{equation*}
and \eqref{eq:l1} follows from \eqref{eq:init-v-eps}.

Using \eqref{ass:w}, \eqref{eq:w'}, \eqref{eq:linftyconv}, and Lemma \ref{lm:linfty}
\begin{align*}
\int_\R(w_\eta*\re)(t,x)dx =& \int_\R\int_x^{x+\eta}w_\eta(y-x)\re(t,y)dydx=  \int_\R\int_0^{\eta}w_\eta(y)\re(t,y+x)dydx\\
=& \norm{w_\eta}_{L^1(\R)}\norm{\re(t,\cdot)}_{L^1(\R)}=\norm{\re(t,\cdot)}_{L^1(\R)},\\
\int_\R|\px(w_\eta*\re)(t,x)|dx\le& \int_\R\int_x^{x+\eta}|w'_\eta(y-x)|\re(t,y)dxdy+w_\eta(0)\int_\R\re dx\\
=&  -\int_\R\int_0^{\eta}w_\eta'(y)\re(t,y+x)dxdy+w_\eta(0)\int_\R\re dx\\
=& 2w_\eta(0)\norm{\re(t,\cdot)}_{L^1(\R)}.
\end{align*}
Therefore, \eqref{eq:linftyconv}, \eqref{eq:l1conv}, and \eqref{eq:BVconv} follow from \eqref{eq:l1}.
\end{proof}

\begin{lemma}[{\bf $BV$ estimate in x}]
\label{lm:BVx}
We have that
\begin{equation*}
\norm{\px\re(t,\cdot)}_{L^1((-\infty,-\delta)\cup(\delta,\infty))}\le C_\delta,
\end{equation*}
for every $t\ge0$ and $\eps,\,\delta>0$ where $C_\delta$ is a constant depending on $\delta$ but not on $\eps$.
\end{lemma}

\begin{proof}
Let us consider the function
\begin{equation*}
\chi(x)=\begin{cases}
1,      \qquad &x\in (-\infty,-2\delta)\,\cup\,(2\delta, +\infty),\\
0, &x\in (-\delta,\delta),
\end{cases}
\end{equation*}
such that 
\begin{align*}
\chi\in C^\infty(\R), &\quad
 0\leq \chi(x)\leq 1, \\
 \chi'(x)\geq 0 \: \text{ for } \: x\in[0,+\infty),&\quad\chi'(x)\leq 0 \: \text{ for } \: x\in(-\infty,0].  
\end{align*}
It is not restrictive to assume $\eps<\delta$. In such a way we have that the supports of $\chi$ and $\ve'$ are disjoint.
Finally, we observe that
\begin{equation*}
\frac{\chi'}{\chi},\frac{\chi''}{\chi}\in L^\infty(\R).
\end{equation*}

Differentiating 
the equation in \eqref{eq:CLeps} w.r.t. the space variable
\begin{align*}
\ptx \re &+ \px\big((1-w_\eta*\re)\ve(x) \px \re\big)+\px\big(\re (1-w_\eta*\re) \ve'(x)\big)\\
&+\px\big(\re (w_\eta'*\re) \ve(x)\big)+w_\eta(0)\px\big(\re^2 \ve(x)\big)=\varepsilon \pxxx \re. 
\end{align*}
Using \cite[Lemma 2]{BLN} and Lemmas \ref{lm:linfty}, and \ref{lm:l1}
\begin{align*}
\frac{d}{dt} \int_\R& \abs{\chi(x)\px\re} d x= \int_\R \chi(x)\ptx\re\sgn{\px\re} d x\\
=&\varepsilon  \int_\R \chi(x)\pxxx \re\sgn{\px\re} d x\\
&-\int_\R \chi(x) \px\big((1-w_\eta*\re)\ve(x) \px \re\big)\sgn{\px\re} d x\\
&\underbrace{-\int_\R \chi(x) \px\big(\re |1-w_\eta*\re| \ve'(x)\big)\sgn{\px\re} d x}_{=0}\\
&-\int_\R \chi(x) \px\big(\re (w_\eta'*\re) \ve(x)\big)\sgn{\px\re} d x\\
&-w_\eta(0)\int_\R \chi(x) \px\big(\re^2 \ve(x)\big)\sgn{\px\re} d x\\
=&\underbrace{-\varepsilon  \int_\R \chi(x)(\pxx \re)^2\delta_{\{\px\re=0\}}}_{\le0}-
\varepsilon  \int_\R \chi'(x)\underbrace{\pxx \re\sgn{\px\re} }_{=\px|\px\re|}d x\\
&+\underbrace{\int_\R \chi(x) (1-w_\eta*\re)\ve(x) \px \re\pxx \re\delta_{\{\px\re=0\}}d x}_{=0}\\
&+\int_\R \chi'(x) (1-w_\eta*\re)\ve(x)|\px \re| d x\\
&-\int_\R \chi(x) (w_\eta'*\re) \ve(x)|\px\re| d x\\
&\underbrace{-\int_\R \chi(x) (\re (w_\eta'*\re) \ve'(x)\sgn{\px\re} d x}_{=0}\\
&+\int_\R \chi(x) \re (w_\eta''*\re) \ve(x)\sgn{\px\re} d x+
w_\eta'(0)\int_\R \chi(x) \re^2 \ve(x)\sgn{\px\re} d x\\
&-2w_\eta(0)\int_\R \chi(x) \re \ve(x)|\px\re| d x-w_\eta(0)\underbrace{\int_\R \chi(x) \re^2 \ve'(x)\sgn{\px\re} d x}_{=0}\\
\le c& \int_\R\chi(x)|\px\re|dx+c\int_\R \re  d x\le c \int_\R\chi(x)|\px\re|dx+c\norm{\rho_0}_{L^1(\R)},
\end{align*}
where $\delta_{\{\px\re=0\}}$ is the Dirac delta concentrated on the set $\{\px\re=0\}$.
Thanks to the Gronwall Lemma  we get
\begin{equation*}
\norm{\chi\px\re(t,\cdot)}_{L^1(\R)}\le e^{ct}\norm{\chi\px\rho_{0,\eps}}_{L^1(\R)}+c(e^{ct}-1),
\end{equation*}
and using \eqref{eq:init-v-eps} we get the claim.
\end{proof}

\begin{lemma}[{\bf Compactness}]
\label{lm:comp}
There exists a function $\rho:[0,\infty)\times\R\to\R$ and a subsequence $\{\eps_k\}_k\subset (0,\infty),\,\epk\to0,$ such that
\begin{align*}
&0\le \rho\le 1,\qquad \rho\in BV((0,\infty)\times((-\infty,-\delta)\cup(\delta,\infty))),\qquad \delta>0,\\
&\rho_\epk\to\rho\qquad \text{a.e. and in $L^p_{loc}((0,\infty)\times\R),\,1\le p<\infty$.}
\end{align*}
\end{lemma}

\begin{proof}
Thanks to Lemma \ref{lm:BVx} the sequence $\{\re \chi_{I_\delta}\}_{\eps,\delta>0}$ of approximate solutions to \eqref{eq:CL}  constructed by vanishing viscosity has uniformly bounded variation on each interval of the type $I_\delta=(-\infty,-\delta)\cup(\delta,+\infty), \, \delta>0.$ 
Moreover, thanks to Lemma \ref{lm:linfty} the $L^\infty-$norm of the sequence $\{\re \chi_{I_\delta}\}_{\eps,\delta>0}$ is bounded by 1.
Thus, applying Helly's Theorem and by a diagonal procedure, we can extract a subsequence $\{\rho_{\varepsilon_k}\chi_{I_{\delta_{k}}}\}_{k\in\N}$  that converges to a function $\rho:[0,\infty)\times \R\to\R$ that satisfies the following conditions: $\rho\in BV((0,\infty)\times((-\infty,-\delta)\cup(\delta,\infty)))$ and $0\leq\rho\leq1,$
\begin{equation*}
\rho_{\varepsilon_k}\chi_{I_{\delta_{k}}} \to \rho \hspace{2em}\text{ a.e. and in }L^p_{loc}((0,\infty)\times \R), \,1\leq p<\infty.
\end{equation*}
Thus, we obtain the compactness of the sequence $\{\rho_{\varepsilon_k}\}_{k\in\N}$ a.e. in $(0,\infty)\times \R$ and for this reason we get the claim.
\end{proof}

\section{Uniqueness and Stability}
\label{sec:uniq}

We are now ready to complete the proof of Theorem \ref{th:main}.

\begin{proof}[Proof of Theorem \ref{th:main}]
The existence of entropy solutions follows using the same arguments of \cite{Andreianovvanishing} and Lemma \ref{lm:comp}.

Let us prove the inequality $\eqref{eq:stability}.$
Following \cite[Theorem 2.1]{KarlsenRisebroTowers}, for any two entropy solutions $u$ and $v$ we can derive the $L^1$ contraction property through the doubling of variables technique:
	\begin{align}\notag
		&\iint_{\R^+\times\R} \left(\modulo{u-v}\phi_t+\sgn(u-v)(f(t,x,u)-f(t,x,v))\phi_x\right) dx dt\\ \label{eq:inequality}
		&\leq K \iint_{\R^+\times\R} \modulo{u-v} \phi dx dt, 
	\end{align}
where $K=K(T),$ for any $0\leq \phi \in \mathcal D{(\R^+\times \R^*)}.$
	We remove the assumption in \eqref{eq:inequality} that $\phi$ vanishes near $0,$ by introducing the following Lipschitz function for $h>0$
	\begin{equation*}
		\mu_h(x)=\begin{cases}
			\frac{1}{h}(x+2h), \quad &x\in[-2h,-h],\\
			1, &x\in[-h,h],\\
			\frac{1}{h}(2h-x), \quad &x\in[h,2h],\\
			0, &\modulo{x}\geq 2h.
		\end{cases}
	\end{equation*}
	Now we can define $\Psi_h(x)=1-\mu_h(x),$ noticing that $\Psi_h \to 1$ in $L^1$ as $h\to 0.$ Moreover, $\Psi_h$ vanishes in a neighborhood of $0.$ For any $0\leq \Phi \in \Cc\infty(\R^+\times\R),$ we can check that $\phi=\Phi\Psi_h$ is an admissible test function for \eqref{eq:inequality}. Using $\phi$ in \eqref{eq:inequality} and integrating by parts we get 
	\begin{align*}
		&\iint_{\R^+\times\R} \left(\modulo{u- v}\Phi_t\Psi_h+\sgn(u- v)(f(t,x,u)-f(t,x,v))\Phi_x\Psi_h \right) dx dt\\ 
		&-\underbrace{\iint_{\R^+\times\R}\sgn(u-v)(f(t,x,u)-f(t,x,v))\Phi(t,x)\Psi'_h(x) dx dt}_{J(h)}\\
		&\leq K \iint_{\R^+\times\R} \modulo{u-v} \Phi \Psi_h dx dt. 
	\end{align*}
	Sending $h\to 0$ we end up with 
	\begin{align*}\notag
		&\iint_{\R^+\times\R} \left(\modulo{u-v}\Phi_t+\sgn(u-v)(f(t,x,u)-f(t,x,v))\Phi_x\right) dx dt\\ \label{eq:inequality}
		&\leq K \iint_{\R^+\times\R} \modulo{u-v} \Phi dx dt+\lim_{h\to 0}J(h). 
	\end{align*}
	We can write 
	\begin{align*}
		\lim_{h\to 0} J(h) &= \lim_{h \to 0} \frac{1}{h} \int_0^T\int_h^{2h} \sgn(u-v)(f(t,x,u)-f(t,x,v)) dx dt\\
		&- \lim_{h \to 0} \frac{1}{h} \int_0^T\int_{-2h}^{-h} \sgn(u-v)(f(t,x,u)-f(t,x,v)) dx dt\\
		&=\int_0^T  [\sgn(u-v)(f(t,x,u)-f(t,x,v))]_{x=0^-}^{x=0^+} \Phi(t,0) dt,
	\end{align*}
	where we indicate the limits from the right and left at $x=0.$
	The aim is to prove that the limit $\displaystyle{\lim_{h \to 0}}J(h)\leq 0.$ This is equivalent to prove that the quantity
	\begin{equation*}
		S:=[\sgn(u-v)(f(t,x,u)-f(t,x,v))]_{x=0^-}^{x=0^+}\leq 0.
	\end{equation*}
	In particular, denoting the right and left traces of $u$ and $v$ with $u_\pm$ and $v_\pm,$ we can write 
	\begin{align*}
	S=&v_r \sgn(u_+-v_+) \left(u_+\left(1-\int_0^\eta u(t,y) \omega_\eta(y) dy\right)-v_+\left(1-\int_0^\eta v(t,y) \omega_\eta(y) dy\right)\right)\\
	&-v_l \sgn(u_--v_-) \left(u_-\left(1-\int_0^\eta u(t,y) \omega_\eta(y) dy\right)-v_-\left(1-\int_0^\eta v(t,y) \omega_\eta(y) dy\right)\right)\\
	=&v_r \sgn(u_+-v_+) (v_+-u_+)\int_0^\eta u(t,y) \omega_\eta(y) dy\\
	&-v_r \sgn(u_+-v_+) u_+\int_0^\eta (u(t,y)-v(t,y) )\omega_\eta(y) dy\\
	&+v_r \abs{u_+-v_+} \\
	&-v_l \sgn(u_--v_-) (v_--u_-)\int_0^\eta u(t,y) \omega_\eta(y) dy\\
	&+v_l \sgn(u_--v_-) v_-\int_0^\eta (u(t,y)-v(t,y) )\omega_\eta(y) dy\\
	&-v_l \abs{u_--v_-} \\
	=&\underbrace{(v_r \abs{u_+-v_+}-v_l \abs{u_--v_-})}_{=0} \left(1-\int_0^\eta u(y,t)\omega_\eta(y)dy\right)\\
	&+\underbrace{(v_r v_+-v_l v_-)}_{=0} \sgn(u_--v_-)\int_0^\eta (v(t,y)-u(t,y)) \omega_\eta(y)dy.\\
  \end{align*}
    A simple application of the Rankine-Hugoniot condition yields $S=0,$ being $u_+=\frac{v_l}{v_r}u_-$ and $v_+=\frac{v_l}{v_r}v_-.$
	In this way we know that \eqref{eq:inequality} holds for any $0\leq\phi\in \Cc\infty (\R^+\times\R).$ For $r>1,$ let $\gamma_r:\R \to \R$ be a $C^\infty$ function which takes values in $[0,1]$ and satisfies 
	\begin{equation*}
		\gamma_r(x)=\begin{cases}
			1, \quad \modulo{x}\leq r,\\
			0, \quad \modulo{x}\geq r+1.
		\end{cases}
	\end{equation*}
	Fix $s_0$ and $s$ such that $0<s_0<s<T.$ For any $\tau>0$ and $k>0$ with $0<s_0+\tau<s+k<T,$ let $\beta_{\tau,k}:[0,T]\to\R$ be a Lipschitz function that is linear on $[s_0,s_0+\tau[\cup[s,s+k]$ and satisfies 
	\begin{equation*}
		\beta_{\tau,k}(t)=\begin{cases}
			0, \quad t\in[0,s_0]\cup[s+k,T],\\
			1,  \quad t\in[s_0+\tau,s].
		\end{cases}
	\end{equation*}
	We can take the admissible test function via a standard regularization argument $\phi=\gamma_r(x)\beta_{\tau,k}(t).$ Using this test function in \eqref{eq:inequality} we obtain 
	\begin{align*}
		\notag
		&\frac{1}{k} \int_s^{s+k} \int_{\R} \modulo{u(t,x)-v(t,x)} \gamma_r(x) dx dt  -\frac{1}{\tau} \int_{s_0}^{s_0+k} \int_\R  \modulo{u(t,x)-v(t,x)} \gamma_r(x) dx dt\\
		&\qquad\leq K \int_{s_0}^{s_0+k} \int_\R \modulo{u-v} \gamma_r(x) dx dt\\
		&\qquad\qquad+\norma{\gamma'_r}_\infty \int_{s_0}^{s+k} \int_{r\leq \modulo{x}\leq r+1}\sgn(u-v)(f(t,x,u)-f(t,x,v)) dx dt. 
	\end{align*}
	Sending $s_0\to 0,$ we get 
	\begin{align*}
		\notag
		\frac{1}{k} &\int_s^{s+k} \int_{-r}^{r} 
		\modulo{u(t,x)-v(t,x)} \gamma_r(x) dx \,dt \\
		\leq & \int_{-r}^{r} \modulo{ u_0(x)- v_0(x)} dx + \frac{1}{\tau} \int_{0}^{\tau} \int_{-r}^{r} \modulo{v(t,x)-v_0(x)} dx dt\\
		&+\frac{1}{\tau} \int_{0}^{\tau} \int_{-r}^{r} \modulo{ u(t,x)-u_0(x)} dx \,dt + K \int_{0}^{t+\tau} \int_\R \modulo{u-v} \gamma_r(x) dx \,dt + o\left(\frac{1}{r}\right).
	\end{align*}
	Observe that the second and the third terms on the right-hand side of the inequality tends to zero as $\tau\to 0$ following the same argument in \cite[Lemma B.1]{KarlsenRisebroTowers}, because our initial condition is satisfied in the ``weak" sense of the definition of our entropy condition. Sending $\tau\to 0$ and $r\to\infty,$ we have 
	\begin{align*}
		\frac{1}{k} \int_s^{s+k} \int_\R \modulo{u(t,x)-v(t,x)} dx \,dt \leq& \int_\R \modulo{u_0(x)-v_0(x)}dx\\
		&+K \int_0^{s+k} \int_\R  \modulo{u(t,x)-v(t,x)} dx \,dt.
	\end{align*}
	Sending $k\to0$ and an application of Gronwall's inequality give us the statement.
\end{proof}

\begin{lemma}[{\bf A Kru\v{z}kov-type integral inequality}]
	For any two entropy solutions $u=u(t,x)$ and $v=v(t,x)$
	the integral inequality \eqref{eq:inequality} holds for any $0\leq \phi\in \Cc\infty(\R^+\times \R\setminus\{0\}).$  
\end{lemma}
\begin{proof}
	Let $0\leq\phi\in \Cc\infty\left((\R^+\times \R\setminus\{0\})^2\right),\: \phi=\phi(t,x,s,y), \: u=u(t,x)$ and $v=v(s,y).$
	From the definition of entropy solution for $u=u(t,x)$ with $\kappa=v(s,y)$ we get 
	\begin{align*}
		&-\iint_{\R^+\times \R} \left(\modulo{u-v}\phi_t+\sgn(u-v)\left(f(t,x,u)-f(t,x,v)\right)\phi_x\right)\, dt\, dx \\
		&+\iint_{\R^+\times \R\setminus\{0\}} \sgn(u-v) f(t,x,v)_x \phi \,dt\, dx \leq 0.
	\end{align*}
	Integrating over $(s,y)\in \R^+\times \R,$ we find
	\begin{equation} 
\label{eq:1ineq}\begin{split}
		&-\iiiint_{(\R^+\times \R)^2} \left(\modulo{u-v}\phi_t+\sgn(u-v)\left(f(t,x,u)-f(t,x,v)\right)\phi_x\right)\, dt \,dx \,ds\, dy\\ 
		&+\iiiint_{(\R^+\times \R\setminus\{0\})^2} \sgn(u-v) f(t,x,v)_x \phi\, dt\, dx\, ds\, dy\leq 0.
	\end{split}\end{equation}
	Similarly, for the entropy solution $v=v(s,y)$ with $\alpha(y)=u(t,x)$
\begin{equation} 
\label{eq:2ineq}\begin{split}
		&-\iiiint_{(\R^+\times \R)^2} \left(\modulo{v-u}\phi_s+\sgn(v-u)\left(f(s,y,v)-f(s,y,u)\right)\phi_x\right)\, dt\, dx\, ds\, dy\\ 
		&+\iiiint_{(\R^+\times \R\setminus\{0\})^2} \sgn(u-v) f(t,x,v)_x \phi\,dt\, dx\, ds\, dy \leq 0.
	\end{split}\end{equation}
	Note that we can write, for each $(t,x)\in \R^+\times \R\setminus\{0\},$
	\begin{align*}
		&\sgn(u-v)(f(t,x,u)-f(t,x,v))  \phi_x -\sgn(u-v) f(t,x,v)_x \phi\\
		&=\sgn(u-v)(f(t,x,u)-f(s,y,v))\phi_x-\sgn(u-v)\left[(f(t,x,v)-f(s,y,v))\phi\right]_x,
	\end{align*}
	so that
	\begin{align*}
		&-\iiiint_{(\R^+\times \R)^2}\sgn(u-v)(f(t,x,u)-f(t,x,v))  \phi_x \,dt\,dx\,ds\,dy \\
		&+\iiiint_{(\R^+\times \R\setminus\{0\})^2}\sgn(u-v) f(t,x,v)_x \phi \,dt\,dx\,ds\,dy\\
		=&-\iiiint_{(\R^+\times \R)^2}\sgn(u-v)(f(t,x,u)-f(s,y,v))\phi_x\,dt\,dx\,ds\,dy\\
		&+ \iiint_{(\R^+\times \R\setminus\{0\})^2} \sgn(u-v)\left[(f(t,x,v)-f(s,y,v))\phi\right]_x \,dt\,dx\,ds\,dy.
	\end{align*}
	Similarly, writing, for each $(y,s)\in \R^+\times \R\setminus\{0\}$
	\begin{align*}
		&\sgn(v-u)(f(s,y,v)-f(s,y,u))  \phi_y -\sgn(v- u) f(s,y,u)_y \phi\\
		&=\sgn(u-v)(f(s,y,v)-f(s,y,u))\phi_y-\sgn(u-v)\left[(f(t,x, u)-f(s,y, u))\phi\right]_x,
	\end{align*}
	so that
	\begin{align*}
		&-\iiiint_{(\R^+\times \R)^2}\sgn(u-v)(f(s,y,v)-f(s,y, u))  \phi_y \,dt\,dx\,ds\,dy \\
		&+\iiiint_{(\R^+\times \R\setminus\{0\})^2}\sgn(u-v) f(s,y, u)_y \phi \,dt\,dx\,ds\,dy\\
		=&-\iiiint_{(\R^+\times \R)^2}\sgn(u-v)(f(t,x,v)-f(s,y,u))\phi_x\,dt\,dx\,ds\,dy\\
		&\hspace{4em}+ \iiint_{(\R^+\times \R\setminus\{0\})^2} \sgn(u-v)\left[(f(t,x, u)-f(s,y, u))\phi\right]_y \,dt\,dx\,ds\,dy.
	\end{align*}
	Let us introduce the notations 
	\begin{align*}
		&\partial_{t+s}=\partial_t+\partial_s, \quad \partial_{x+y}=\partial_x+\partial_y,\\
		&\partial^2_{x+y}=(\partial_x+\partial_y)^2=\partial_x^2+2\partial_x\partial_y+\partial_y^2.
	\end{align*}
	Adding \eqref{eq:1ineq} and \eqref{eq:2ineq} we obtain 
	\begin{equation}
		\label{eq:A22}\begin{split}
		&-\iiiint_{(\R^+\times \R)^2} \left(\modulo{u-v}\partial_{t+s}\phi+\sgn(u-v)\left(f(t,x,u)-f(s,y,v)\right)\partial_{x+y} \phi\right)   \,dt\,dx\,ds\,dy\\ 
		&+\iiiint_{\R^+\times \R\setminus\{0\}} \sgn(u-v) \left(\partial_x\left[ (f(t,x,v)-f(s,y,v)) \phi \right] 
		\right.\\  
		&\hspace{6em}\left.+ \partial_y \left[(f(t,x,u)-f(s,y,u)) \phi \right]\right)  \,dt\,dx\,ds\,dy \leq 0. 
	\end{split}\end{equation}
	We introduce a non-negative function $\delta\in \Cc\infty(\R),$ satisfying $\delta(\sigma)=\delta(-\sigma),\, \delta(\sigma)=0$ for $\modulo{\sigma}\geq 1,$ and $\int_{\R} \delta(\sigma) d \sigma=1.$ For $u>0$ and $z\in\R,$ let $\delta_{p}(z)=\frac{1}{p}\delta(\frac{z}{p}).$ We take our test function $\phi=\phi(t,x,s,y)$ to be of the form 
	\begin{equation*}
		\Phi(t,x,s,y)= \phi\left(\frac{t+s}{2},\frac{x+y}{2}\right)\delta_p\left(\frac{x-y}{2}\right) \delta_p\left(\frac{t-s}{2} \right),
	\end{equation*}
	where $0\leq \phi\in \Cc\infty\left(\R^+\times \R\setminus{0}\right)$ satisfies  
	\begin{equation*}
		\phi(t,x)=0, \quad \forall(t,x)\in[-h,h]\times[0,T],
	\end{equation*}
	for small $h>0.$ By making sure that 
	\begin{equation*}
		p<h,
	\end{equation*}
	one can check that $\Phi$ belongs to $\Cc\infty\left(\left(\R^+\times \R\setminus\{0\}\right)^2\right).$
	We have 
	\begin{align*}
		\partial_{t+s} \Phi(t,x,s,y)= 	\partial_{t+s} \phi\left( \frac{t+s}{2},\frac{x+y}{2}\right)\delta_p\left(\frac{x-y}{2}\right) \delta_p\left(\frac{t-s}{2} \right),\\
		\partial_{x+y} \Phi(t,x,s,y)= 	\partial_{x+y} \phi\left(\frac{t+s}{2},\frac{x+y}{2}\right)\delta_p\left(\frac{x-y}{2}\right) \delta_p\left(\frac{t-s}{2} \right),\\
		\end{align*}
	Using $\Phi$ as test function in \eqref{eq:A22} 
	\begin{align*}
		&-\iiiint_{(\R^+\times \R)^2} \left(I_1(t,x,s,y)+I_2(t,x,s,y)\right)\delta_p\left(\frac{x-y}{2}\right) \delta_p\left(\frac{t-s}{2} \right) \,dt\,dx\,ds\,dy\\
		&\leq \iiiint_{(\R^+\times \R\setminus\{0\})^2} \left(I_3(t,x,s,y)+I_4(t,x,s,y)+I_5(t,x,s,y)\right) \,dt\,dx\,ds\,dy,
	\end{align*}
	where 
	\begin{align*}
		I_1&=\modulo{u(t,x)-v(s,y)} \partial_{t+s} \phi\left(\frac{t+s}{2},\frac{x+y}{2}\right),\\
		I_2&=\sgn{(u(t,x)-v(s,y))} (f(t,x,u)-f(s,y,v)) \partial_{x+y} \phi\left(\frac{t+s}{2}, \frac{x+y}{2}\right),\\
		I_3&=-\sgn{(u(t,x)-v(s,y))} \left(\partial_x f(t,x,v) -\partial_y f(s,y,u)\right) \\
		&\hspace{4em}\phi\left(\frac{t+s}{2},\frac{x+y}{2},\right)\delta_p\left(\frac{x-y}{2}\right) \delta_p\left(\frac{t-s}{2} \right),\\
		I_4&=-\sgn{(u(t,x)-v(s,y))} \delta_p\left(\frac{x-y}{2}\right) \delta_p\left(\frac{t-s}{2} \right)\\
&\hspace{4em}		\left[ \partial_x\phi\left(\frac{t+s}{2},\frac{x+y}{2},\right) (f(t,x,v) -f(s,y, v))\right.\\
&\hspace{4em}	\hspace{4em}		\left.		 \partial_y\phi\left(\frac{t+s}{2},\frac{x+y}{2}\right) (f(t,x, u) -f(s,y,u))\right],\\
		I_5&=\left(F(x,u(t,x),v(s,y))-F(y,u(t,x),v(s,y))\right) \\
		&\hspace{4em}\phi\left(\frac{t+s}{2},\frac{x+y}{2}\right)\partial_x\delta_p\left(\frac{x-y}{2}\right) \delta_p\left(\frac{t-s}{2} \right),
	\end{align*}
where $F(x,u,c):=\sgn{(u-c)}\left(f(t,x,u)-f(t,x,c)\right).$\\
	We now use the change of variables 
	\begin{align*}
		\tilde x =\frac{x+y}{2}, \quad 	\tilde t =\frac{t+s}{2}, \quad z =\frac{x-y}{2},\quad \tau =\frac{t-s}{2},
	\end{align*}
	which maps $(\R^+\times \R)^2$ in $\Omega\subset \R^4$  and $(\R^+\times \R\setminus\{0\})^2$ in $\Omega_0\subset \R^4,$ where
	\begin{align*}
	&	\Omega=\{(\tilde x, \tilde t,z, \tau)\in\R^4 :\,0<\tilde t \pm \tau<T\}, \\
     & \Omega_0=\{(\tilde x, \tilde t,z, \tau)\in\Omega : \tilde x \pm z\neq 0\},
	\end{align*}
	resepectively. 
	With this changes of variables, 
	\begin{align*}
		&\partial_{t+s} \phi\left(\frac{t+s}{2},\frac{x+y}{2}\right)=\partial_{\tilde t} \phi(\tilde t, \tilde x),\\
		&\partial_{x+y} \phi\left(\frac{t+s}{2},\frac{x+y}{2}\right)=\partial_{\tilde x} \phi(\tilde t, \tilde x).	\end{align*}
	Now we can write 
	\begin{align*}
		&-\iiiint_{\Omega} \left(I_1(\tilde t, \tilde x,\tau,z)+I_2(\tilde t, \tilde x,\tau,z)\right)\delta_p\left(z\right) \delta_p\left(\tau \right) \,d\tilde t\,d \tilde x\,d\tau\,dz\\
		&\hspace{4em}\leq \iiiint_{\Omega_0} \left(I_3(\tilde t, \tilde x,\tau,z)+I_4(\tilde t, \tilde x,\tau,z)+I_5(\tilde t, \tilde x,\tau,z)\right) \,d\tilde t\,d \tilde x\,d\tau\,dz,
	\end{align*}
	where
	\begin{align*}
		I_1(\tilde t, \tilde x,\tau,z)=&\modulo{u(\tilde t+\tau,\tilde x+z)-v(\tilde t-\tau,\tilde x-z)} \partial_{\tilde t} \phi\left(\tilde t,\tilde x\right),\\
		I_2(\tilde t, \tilde x,\tau,z)=&\sgn{(u(\tilde t+\tau,\tilde x+z)-v(\tilde t-\tau,\tilde x-z))}\\
		& (f(\tilde t+\tau,\tilde x+z,u)-f(\tilde t-\tau,\tilde x-z,v)) \partial_{\tilde x} \phi\left(\tilde t,\tilde x\right),\\
		I_3(\tilde t, \tilde x,\tau,z)=&-\sgn{(u(\tilde t+\tau,\tilde x+z)-v(\tilde t-\tau,\tilde x-z))} \\
		&\left(\partial_{\tilde x+z} f(\tilde t+\tau,\tilde x+z,v) -\partial_{\tilde x -z} f(\tilde t-\tau,\tilde x-z,u)\right)
		\phi\left(\tilde t,\tilde x\right)\delta_p\left(z\right) \delta_p\left(\tau \right),\\
		I_4(\tilde t, \tilde x,\tau,z)=&-\sgn{(u(\tilde t+\tau,\tilde x+z)-v(\tilde t-\tau,\tilde x-z))}\\
		&  \partial_{\tilde x}\phi\left(\tilde t,\tilde x\right) \delta_p\left(z\right) \delta_p\left(\tau \right)
		\left[ (f(\tilde t+\tau,\tilde x+z,v) -f(\tilde t-\tau,\tilde x-z, v))\right.\\ 
		&\hspace{8em}\left.+(f(\tilde t+\tau,\tilde x+z, u) -f(\tilde t-\tau,\tilde x-z,u))\right],\\
		I_5(\tilde t, \tilde x,\tau,z)=&\left(F(\tilde x+z,u(\tilde t+\tau,\tilde x+z),v(\tilde t-\tau,\tilde x-z))\right.\\
		&\left.-F( \tilde x-z,u(\tilde t+\tau,\tilde x+z),v(\tilde t-\tau,\tilde x-z))\right)
		\phi\left(\tilde t, \tilde x\right) \partial_z\delta_p\left(z\right) \delta_p\left(\tau \right).
	\end{align*}
	Employing Lebesgue's differentiation theorem, to obtain the following limits
	\begin{align*}
		&\lim_{p\to 0} \iiiint_{\Omega} I_1(\tilde t, \tilde x,\tau, z) \delta_{p}(z)\delta_{p}(\tau)\, d\tilde t\, d \tilde x\, d\tau\, dz\\
		&\hspace{7em}= \iint_{\R^+\times \R} \modulo{u(t,x)-v(t,x)} \partial_t \phi(t,x) dt dx,\\
		&\lim_{p\to 0} \iiiint_{\Omega} I_2(\tilde t, \tilde x, \tau, z) \delta_{p}(z)\delta_{p}(\tau)\, d\tilde t\, d \tilde x\, d\tau\, dz\\
		&\hspace{7em}= \iint_{\R^+\times \R} \sgn{(u(t,x)-v(t,x))(f(t,x,u)-f(t,x,v))} \partial_x \phi(t,x) dt dx.	\end{align*}
	Let us consider the term $I_3.$ Note that $I_3(\tilde t, \tilde x,\tau, z)=0$, if $\tilde x\in[-h,h],$ since then $\phi(\tilde t,\tilde x)=0$ for any $\tilde t,$ or if $\modulo{z}\geq p.$ On the other hand, if $\tilde x \not\in [-h,h],$ then $\tilde x\pm z<0$ or $\tilde x\pm z>0,$ at least when $\modulo{z}<p$ and $p<h.$
	Defining $U(t,x)=1-\omega_\eta\ast u$ and $V(t,x)=1-\omega_\eta\ast v,$  and sending $p\to0:$
	\begin{align*}
	&\lim_{p\to 0} \iiiint_{\Omega_0} I_3(\tilde t, \tilde x,  \tau, z) \, d\tilde t\, d \tilde x\, d\tau\, dz\\
	&=\iint_{\R^+\times \R\setminus\{0\}} \sgn{(u(t,x)-v(t,x))} \mathfrak{v}(x) \left(v \partial_x V- u \partial_x U\right) \phi\left(t,x\right) \, dt \, dx\\
	&\leq v_r\norma{\partial_x V} \iint_{\R^+\times \R\setminus\{0\}} \modulo{u-v} \phi(t,x)\, dt \, dx + v_r \iint_{\R^+\times \R\setminus\{0\}}\modulo{\rho} \modulo{\partial_x V- \partial_x U}\, dt \, dx\\
	&\leq  K_1 \iint_{\R^+\times \R\setminus\{0\}} \modulo{u-v} \phi(t,x)\, dt \, dx. 
		\end{align*}
	In fact, 
	\begin{align*}
	\modulo{\partial_{x}V-\partial_x U}\leq& \norma{\omega'_\eta}\,  \,\norma{u(t,\cdot)-v(t,\cdot)}_{L^1}\\&+\omega_\eta(0)
	\left( \modulo{u- v}(t,x+\eta)+ \modulo{u-v}(t,x) \right).
	 \end{align*}
	 	The term $I_4$ converges to zero as $p\to0.$
	 	Finally, the term $I_5$ 
	 	\begin{equation*}
	 	\lim_{p\to 0} \iiiint_{\Omega_0} I_5(\tilde t, \tilde x, \tau, z) \, d\tilde t\, d \tilde x\, d\tau\, dz\leq  K_2 \iint_{\R^+\times \R\setminus\{0\}} \modulo{u-v} \phi(t,x)\, dt \, dx. 
	 	\end{equation*}
 	\end{proof}

\bibliographystyle{abbrv}

\end{document}